\documentclass[11pt]{article}
\usepackage{geometry} 
\usepackage{amsmath,amsthm,amssymb}                
\usepackage{graphicx, color}
\usepackage{amssymb}
\usepackage{epstopdf}
\usepackage{pdfsync}
\usepackage{mathabx}

\newtheorem{definition}{Definition}[section]
\newtheorem{lemma}[definition]{Lemma}
\newtheorem{theorem}[definition]{Theorem}
\newtheorem{proposition}[definition]{Proposition}

\newtheorem{remark}[definition]{Remark}

\def\e{\varepsilon}

\newcommand{\restr}{\ \rule{.4pt}{7pt}\rule{6pt}{.4pt}\ }

\def\dxy{\,dx dy}
\def\dx{\,dx}

\title{Nonlocal-interaction vortices}
\author{Margherita Solci\\ 
{\small DADU, Universit\`a di Sassari}\\ {\small piazza Duomo 6, 07041 Alghero (Italy)}}
\date{}
\begin{document} 

\maketitle

\begin{abstract} We consider sequences of quadratic non-local functionals, depending on a small parameter $\e$, that approximate the Dirichlet integral by a well-known result by Bourgain, Brezis and Mironescu. Similarly to what is done for hard-core approximations to vortex energies in the case of the Dirichlet integral, we further scale such energies by $|\log\e|^{-1}$ and restrict them to $S^1$-valued functions. We introduce a notion of convergence of functions to integral currents with respect to which such energies are equi-coercive, and show the converge to a vortex energy, similarly to the limit behaviour of Ginzburg-Landau energies at the vortex scaling.


\end{abstract}

\section{Introduction}
In their seminal paper \cite{BBM}, Bourgain, Brezis and Mironescu have shown, among many other results, that sequences of nonlocal energies depending on a vanishing  parameter $\e$ approximate Sobolev seminorms. In particular if $\Omega$ is an open subset of $\mathbb R^d$ and $\rho$ is an integrable positive kernel with compact support, their result implies that the $\Gamma$-limit of 
energies of the form
\begin{equation}\label{effe-e-bbM} F^{BBM}_\e(u)=
{1\over\e^{d+2}} \int_{\Omega\times\Omega}\rho\Bigl({|x-y|\over\e}\Bigr) {|u(x)-u(y)|^2}\dxy,
\end{equation}
defined on functions $u\in L^1_{\rm loc} (\Omega)$, is an explicit constant depending only on $\rho$ and the dimension $d$ times the Dirichlet integral. For a general variational approach to convolution-type energies modeled on functionals \eqref{effe-e-bbM} we refer to \cite{AABPT}. Note that we may rewrite such type of functionals in the form 
\begin{equation}\label{BBM2} F^{BBM}_\e(u)=
{1\over\e^d} \int_{\Omega\times\Omega}\rho\Bigl({|x-y|\over\e}\Bigr) {|u(x)-u(y)|^2\over|x-y|^2}\dxy,
\end{equation}
up to considering the kernel $\rho(|\xi|)|\xi|^2$ instead of $\rho$ in \eqref{effe-e-bbM}, which is a notation closer to that of \cite{BBM}.

Functionals such as those considered above, and more in general functionals of nonlocal type, have recently been used to study a number of variational problems, where the role of gradient terms is played by either finite differences, or fractional seminorms, or nonlocal gradients. This has been used to give non-local approximation of elastic energies or sharp-interface theories often using a singular-perturbation approach (see {\em e.g.}~\cite{ABCP,BMP,CS,DPV,MD}). Our objective is to extend the use of nonlocal energies in order to treat models with topological singularities such as those arising in the asymptotic analysis of Ginzburg-Landau energies \cite{BBH,SS}.  To that end, we consider an open domain $\Omega\subset\mathbb R^d$ with $d\geq 2$ and sequences of nonlocal-interaction energies defined for vector functions $u\in L^1_{\rm loc} (\Omega; \mathbb R^2)$ with $|u(x)|=1$ almost everywhere as
\begin{equation}\label{effe-e} F_\e(u)=
{1\over\e^{d+2}|\log\e|} \int_{\Omega\times\Omega}\rho\Bigl({|x-y|\over\e}\Bigr) {|u(x)-u(y)|^2}\dxy. 
\end{equation} 
In light of the Bourgain, Brezis and Mironescu result, heuristically this can be regarded as a singular perturbation of Dirichlet energies with the constraint $|u|=1$. 
This approach can be compared with the {\em hard-core} approach to sequences of energies of Ginzburg-Landau type at the {\em vortex scaling} (see \cite{BBH,AP}). 
In that case, one considers scaled Dirichlet integrals on a domain where we have removed from $\Omega$ a number of small balls close to possible singularities. 
In our case, contrary to the Sobolev-space hard-core model, it is not necessary to remove such small balls  since no regularity on $u$ is needed.
We will show that energies $F_\e$ are approximated as $\e\to0$ by a vortex energy defined on ($\pi$ times) integral $(d-2)$-currents, as in the case of Ginzburg-Landau functionals \cite{BBH,J,JS,SS}. 
A similar approximation of vortex energies has been shown in \cite{AC} using lattice energies, proving a connection with the so-called $XY$ model. In a sense energies \eqref{effe-e} are a formal continuous version of those lattice energies, with a major difference, that in \eqref{effe-e} we do not assume any regularity on the functions $u$, while lattice energies are essentially defined on piecewise-affine functions, with restrictions due to the lattice structure. As a consequence, we have more flexibility in the use of functionals \eqref{effe-e} in dimension $d\ge 3$ (see \cite[Section 4.3.2]{AC} for a discussion on technical issues in the lattice case).

\smallskip 

A loose explanation of our asymptotic result can be given as follows. For ease of notation we consider the case $d=2$, in which case integral $(d-2)$-currents in $\Omega$ are measures of the form 
\begin{equation}
\mu=\pi \sum_{i=1}^N d_i \delta_{x_i}, \qquad d_i\in \mathbb Z \hbox{ and } x_i\in\Omega,
\end{equation}
and we take $\Omega=\mathbb R^2$ in order to simplify changes of variables. Using the change of variables $\xi={x-y\over \e}$ we rewrite the energies as
\begin{equation}
\int_{\mathbb R^2}{\rho(|\xi|)} \Bigl({1\over\e^2|\log\e|}\int_{\mathbb R^2} |u(x+\e\xi)-u(x)|^2\dx\Bigr)\,d\xi.
\end{equation}
The inner integral can be analyzed by discretization. In the case $\xi=|\xi|e_1$ we can estimate
\begin{eqnarray*}
&& \hskip-3cm{1\over\e^2|\log\e|}\int_{\mathbb R^2} |u(x+\e\xi)-u(x)|^2\dx\\&=&{1\over\e^2|\log\e|}
\sum_{k\in\mathbb Z^2} 
\int_{\e |\xi|(k+ [0,1]^2)} |u(x+\e\xi)-u(x)|^2\dx\\
&\ge &{|\xi|^2\over|\log\e|}
\sum_{k\in\mathbb Z^2} \Bigl|{1\over\e^2|\xi|^2}
\int_{\e |\xi|(k+ [0,1]^2)} (u(x+\e\xi)-u(x))\dx\Bigr|^2\\
&=&{|\xi|^2\over|\log\e|}
\sum_{k\in\mathbb Z^2} | u^{\e,\xi}_{k+e_1}-u^{\e,\xi}_k|^2,
\end{eqnarray*}
where 
$$
u^{\e,\xi}_k= {1\over\e^2 |\xi|^2} \int_{\e |\xi|(k+ [0,1]^2)} u(x)\dx.
$$
Repeating the same argument for $\xi=|\xi|e_2$ and taking into account double counting, we can use the discrete energies $X_\e(u^{\e,\xi})$
in the lower bound, with
\begin{equation}\label{XY}
X_\e(v)={1\over |\log\e|}
\sum_{\langle i,j\rangle} |v_i-v_j|^2,
\end{equation}
where $v\colon\e\mathbb Z^2\to \mathbb R^2$, $v_i=v(\e i)$, and $\langle i,j\rangle$ denotes summation of nearest neighbours in $\mathbb Z^2$. An analog estimate holds for arbitrary $\xi$ upon using suitable discretization on rotated lattices. In \cite{AC} the analysis of energies of form \eqref{XY} when $v_i\in S^1$ has been carried out, showing their equi-coerciveness with respect to the flat convergence of the Jacobians of interpolations, implying that the limit energies are defined on vortex measures. In the heuristic argument just illustrated the results of \cite{AC} cannot directly be applied, since the averages $u^{\e,\xi}_k$ in general do not lie in $S^1$. Furthermore, in principle the limit measure $\mu$ might depend on $\xi$. In order to overcome these structural complexities, a Compactness Lemma will ensure that indeed the discrete functions $u^{\e,\xi}$ do converge, up to subsequences, in the same sense as that obtained in \cite{AC} to a vortex measure independent of $\xi$. This result will allow to give a definition of convergence of $u_\e$ to $\mu$ even though the functions $u_\e$ themselves may be not weakly differentiable. Note that the compactness result will be proved using the fact that the piecewise-affine interpolation of $u^{\e,\xi}$ can be seen as an average of piecewise-affine interpolation of $S^1$-valued maps to almost all of which we may apply a compactness argument.
 
Once a compactness result is achieved, we compute the $\Gamma$-limit with respect to the convergence above. Proceeding in the heuristic argument, we can use Fatou's Lemma to give a lower bound as 
\begin{eqnarray*}
\liminf_{\e\to 0} F_\e(u_\e) \ge {1\over 4} \int_{\mathbb R^2}{\rho(|\xi|)|\xi|^2}  \liminf_{\e\to0} X_\e(u^{\e,\xi})d\xi, \end{eqnarray*}
where the factor $4$ comes from the fact that  $\xi$,   $-\xi$, $\xi^\perp$, and  $-\xi^\perp$ are all taken into account in the definition of $X_\e$. The asymptotic analysis in \cite{AC} ensures that $X_\e$ $\Gamma$-converge to $4\pi \sum_i|d_i|$ ($d_i$ being the coefficients of the limit vortex measure $\mu$ related to $u_\e$), when the discrete functions are $S^1$-valued. If this were true also for $u^{\e,\xi}$ then we would conclude that
\begin{eqnarray*}
\liminf_{\e\to 0} F_\e(u_\e) \ge {\pi} \int_{\mathbb R^2}{\rho(|\xi|)|\xi|^2}d\xi \, \sum_i|d_i|.
\end{eqnarray*}
Again, as in the compactness argument, this can be proven using the fact that the piecewise-affine interpolation of $u^{\e,\xi}$ can be seen as an average of piecewise-affine interpolation of $S^1$-valued maps.
The optimality of the lower bound can be then shown by a direct computation. It is worth mentioning that, while recovery sequences in the case of Ginzburg-Landau energies use a smoothening argument, and the discrete approach uses suitable interpolations of the function  ${x\over|x|}$, whose singularity gives a simple vortex in $0$, in our case this function can be directly used as a recovery sequence.

In the paper all results are proven for  energies \eqref{effe-e} in the general $d$-dimensional case, in which the limit is defined on currents of the form $\mu=\pi M$, where $M$ is a $(d-2)$-integral current. The form of the $\Gamma$-limit is then  
 $$
 F(\mu)={2\pi\over d}\|M\|\int_{\mathbb R^d}{\rho(|\xi|)|\xi|^2}d\xi,
 $$
with $\|M\|$ the mass of the current in $\Omega$. Note that $\|M\|=\sum_i|d_i|$ in the $2$-dimensional case.
 
\medskip 

The plan of the paper is as follows. 
In Section \ref{prel} we introduce the necessary notions about currents, in order to define the space of integral currents in which we frame our asymptotic analysis. In the same section, after recalling the results obtained in 
\cite{AC} for the $XY$ model, we give the main definition of convergence of functions to integral currents via the flat convergence of the Jacobians of the interpolation of the averages, and we finally  
state our main compactness and $\Gamma$-convergence result. In Section \ref{comp-sec} we prove the key compactness lemma, which states the equi-coerciveness of energies $F_\e$ with respect to the convergence of the averages introduced in the previous section. The proof of the $\Gamma$-convergence part of the main theorem is given in Section \ref{conv-sec}, subdivided into the proof of the lower and upper bounds.

\section{Preliminaries and statement of the main result}\label{prel} 
In the paper we will consider problems in $\mathbb R^d$ with $d\geq 2$, even if some of the notions and results presented in this preliminary sections are interesting also in dimension $1$. 

The scalar product in $\mathbb R^d$ is denoted by $\langle\cdot,\cdot\rangle$, the elements of the standard basis are denoted by $e_n$ for  $n\in\{1,\dots, d\}$. The Lebesgue measure in $\mathbb R^d$ is denoted by $\mathcal L^d$; $\mathcal H^k$ denotes the $k$-dimensional Hausdorff measure.  
The letter $C$ will denote a generic positive constant independent of
fixed parameters, whose value may vary from line to line. 

\subsection{Currents}
We will use some terminology, notions and results from the theory of currents. 
Here we only recall some basic definitions, referring {\em e.g.}~to 
the monographs \cite{fed,GMS,simon} for the general theory, 
and to the ample treatment in the works \cite{A,ABS,J,JS} for a very detailed description of their applications to variational problems of Ginzburg-Landau type. In general, in our presentation we borrow the notation from \cite[Chapter 6]{simon}. 
Since we will not directly use fine properties of currents, we only introduce the concepts 
that will be needed to state the relevant compactness results, and refer to the cited works for complete results and references.

Let $\Omega\subset\mathbb R^d$ be an open bounded set with $\mathcal L^d(\partial\Omega)=0$.   
For $h\in\{1,\dots, d\}$, an $h$-form $\omega$ of class $C^\infty$ with compact support in $\Omega$ is 
$\omega=\sum_{\alpha}\phi_\alpha dx_{i_1}\wedge\dots\wedge dx_{i_h}$, 
where the sum is taken over $\alpha=(i_1,\dots, i_h)$ multi-index such that $1\leq i_1\leq\dots\leq i_h\leq d$ and each $\phi_\alpha$ belongs to $C^\infty_{\rm c}(\Omega)$; that is, the space of smooth functions with compact support in $\Omega$.  
The symbol $\wedge$ denotes the external product. 
A $0$-form of class $C^\infty$ with compact support in $\Omega$ is simply a function $\omega\in C^{\infty}_{\rm c}(\Omega)$.
For any $h\in\{0,\dots, d\}$, 
the space 
of $h$-dimensional currents 
is identified with the dual space of the space of the $h$-forms of class $C^\infty$ with compact support in $\Omega$. 
The restriction of a current $T
$ to an open set $U\subset\Omega$ is the current $T\restr U$ acting as $T$ on $h$-forms in $C^\infty_{\rm c}(U)$; that is, $(T\restr U)[\omega]=T[\omega]$ for any $\omega$ $h$-form in $C^\infty_{\rm c}(U)$. 

Following Stokes' Theorem, the boundary of an $h$-current $T$ is an $(h-1)$-current denoted by $\partial T$ and defined by setting 
$\partial T[\omega]=T[d\omega]$ for every $(h-1)$-form of class $C^\infty$ with compact support in $\Omega$, $d\omega$ being the differential of $\omega$.  
We say that an $h$-dimensional current $T$ is a boundary if there exists a $(h+1)$-dimensional current $M$ such that $T=\partial M$.
Moreover, we say that a current $T$ is a boundary locally in $\Omega$ 
if the restriction $T\restr U$ 
is a boundary for any $U\subset\subset\Omega$.  

A current $T$ is said to have (locally) finite mass
if it can be represented as a (locally) bounded Borel measure. 
In this case, the mass $\|T\|$ of the current $T$ is defined as $|T|(\Omega)$, where $|T|$ is 
the variation of the measure. If $U$ is an open subset of $\Omega$, the mass 
of $T$ in $U$ is defined as $\|T\|(U)=|T|(U)$. 

In the sequel, we will use the class of {\em integral currents}. 
To briefly introduce this class, we need to recall some preliminary notions. A set $M\subset \Omega$ is $h$-rectifiable 
if it can be covered by a countable union of $h$-dimensional surfaces of class $C^1$, up to a $\mathcal H^h$-negligible set. 
To define an orientation of such a set, we recall that an $h$-vector $v$ is  
$v=\sum_{\alpha} a_\alpha e_{i_1}\wedge \dots\wedge e_{i_h}$ 
with $\alpha=(i_1,\dots, i_h)$ multi-index as above and $a_\alpha\in\mathbb R$. 
An $h$-covector is given in the corresponding way, by considering the canonical basis of the dual space of $\mathbb R^d$ instead of $\{e_1,\dots,e_d\}$. Note that an $h$-form $\omega$ 
is then a map from $\Omega$ to the set the $h$-covectors.  
An $h$-vector $v$ is simple if it can be expressed as the external product of $h$ vectors in $\mathbb R^d$; that is, $v=v_1\wedge\dots\wedge v_n$, $v_n\in\mathbb R^d$. An orientation of  an $h$-rectifiable set $M$ is a $\mathcal H^h$-measurable map $\tau_M$ defined in $M$ and valued in the space of $h$-vectors such that for $\mathcal H^h$-almost every $x$ $\tau_M(x)=\tau^1_M(x)\wedge\dots\wedge\tau^h_M(x)$ is a simple $h$-vector such that $\{\tau^j_M(x)\}$ is an orthonormal basis for the tangent space to $M$ at $x$. The notation $\langle\!\langle\cdot,\cdot \rangle\!\rangle$ denotes the duality product between $h$-covectors and $h$-vectors in $\mathbb R^d$.  
A $h$-current $T$ is {\em rectifiable} if it can be represented as 
\begin{equation}\label{rc}
T[\omega]=\int_M \sigma(x) \langle\!\langle \omega(x), \tau_M(x) \rangle\!\rangle\, d\mathcal H^h(x) 
\end{equation}
for any smooth $h$-form $\omega$ with compact support, with $M$ an $h$-rectifiable set, $\tau_M$ an orientation of $M$ and 
$\sigma$ is locally summable with respect to the measure $\mathcal H^h\restr M$ and it is integer-valued. 
The function $\sigma$ is called a {\em multiplicity}. Note that \eqref{rc} corresponds to say that $T$ can be represented as the measure $\sigma \tau_M \mathcal H^h\restr M$ associated to the density function $\sigma\tau_M\in L^1_{\rm loc}(\mathcal H^h\restr M)$. In this case, $\|T\|=\int_{M}|\sigma(x)|d\mathcal H^h(x)$; that is, the mass of $T$ is the measure of $M$ counted with respect to the multiplicity. 
A current $T$ is an {\em integral current} if both $T$ and $\partial T$ are rectifiable. 
Note that a rectifiable current $T$ which is a boundary is an integral current, since $\partial T=0$. Such a current will be called an {\em integral boundary}. 
  
An {\em integral polyhedral $h$-current in $\mathbb R^d$} is a finite sum of $h$-currents in $\mathbb R^d$ associated, as in \eqref{rc}, to $h$-dimensional simplices in $\mathbb R^d$ with constant orientations and corresponding integer-valued multiplicities. Polyhedral integral currents in $\Omega$ are defined by restriction. 

In the following we will be mainly interested in currents which are boundaries, and we will use properties of the {\em flat norm} ${\bf F}_{\Omega}(T)$ of a $h$-current $T$ in $\Omega$. The flat norm is defined by setting 
$${\bf F}_{\Omega}(T)=\inf\{\|S\|: S \ (h+1)\hbox{\rm-current in $\Omega $ such that } \partial S=T\},$$ 
where the infimum is $+\infty$ if the current $T$ is not a boundary. 
The convergence with respect to this norm will be used in our results. 
In particular, in the proof of the $\limsup$ inequality we will use a density result for integral polyhedral boundaries with respect to the convergence induced by the flat norm. We will state this density result 
where needed (see Proposition \ref{polydens}). 

Another notion which will be used to define the convergence of sequences of functions is the identification of the Jacobian of a function $u=(u_1,u_1)\in W^{1,2}(\Omega;\mathbb R^2)$, defined in an open domain of $\mathbb R^d$, $d\geq 2$, with a $(d-2)$-current. Indeed, the Jacobian $Ju$ is the $2$-form given by $Ju=du_1\wedge du_2$, where $du_i=\sum_{j=1}^d D_ju_i dx_j$ is the differential of the $i$-th component of $u$. 
Following \cite{ABS} (see also \cite{JS}), we identify vectors and covectors in $\mathbb R^d$ by using the operator $\star$, which maps an $h$-covector $\omega$ in the $(d-h)$-vector $\star\omega$ determined  by the identity 
$\langle\!\langle\omega^\prime,\star\omega\rangle\!\rangle=\langle\!\langle \omega^\prime\wedge\omega, e_1\wedge\dots\wedge e_d \rangle\!\rangle$  
for every $(d-h)$-covector $\omega^\prime$ in $\mathbb R^d$. 
Hence, $Ju$ can be identified with the $(d-2)$-current $\star Ju$ given by 
$$\star Ju[\omega]=\int_\Omega\langle\!\langle \omega\wedge Ju, e_1\wedge\dots\wedge e_d\rangle\!\rangle\, dx$$ 
for every $(d-2)$-form $\omega$ of class $C^\infty_{\rm c}(\Omega)$, where $e_1\wedge\dots\wedge e_d$ is the standard orientation of $\Omega$. Note that since the $2$-form $Ju$ is a differential, then $\star Ju$ is a boundary.  

\smallskip
 The first technical tool is the following lemma concerning the behaviour of Jacobians interpreted as $(d-2)$-currents. 
Note that in the case of functions defined on subsets of $\mathbb R^2$ the convergence of the Jacobians could be directly stated in terms of weak derivatives.

\begin{lemma}[a criterion for equivalent flat-converging sequences {\cite[Lemma 1]{AC}}]\label{lemma1}
Let $U$ be a bounded open set in $\mathbb R^d$ with $d\geq 2$ and let $\{u_\e\}$ and $\{v_\e\}$ be two sequences
belonging to $W^{1,2}(U,\mathbb R^2)$. If there exists a constant $C>0$ such that

\smallskip
{\rm(i)} $\displaystyle \int_U|u_\e-v_\e|^2dx\le C\e^2|\log\e|$;

\smallskip
{\rm(ii)} $\displaystyle \int_U|\nabla u_\e-\nabla v_\e|^2dx\le C|\log\e|$;

\smallskip\noindent
then $\lim\limits_{\e\to0}{\bf F}_U (\star J(u_\e)-\star J(v_\e))=0$.
\end{lemma}

\subsection{Discrete energies}
We will use a compactness and a convergence result for discrete energies defined on $\mathbb Z^d$.

Let $\Omega$ be a bounded open Lipschitz set in $\mathbb R^d$.
For $u\colon\Omega\cap\e\mathbb Z^d\to S^1$ and $U$ open subset of $\Omega$ we define 
\begin{equation}
X_\e(u;U)={1\over|\log\e|}\sum_{\langle i,j\rangle}\e^{d-2}|u_i-u_j|^2,
\end{equation}
where $u_i=u(\e i)$, and $\langle i,j\rangle$ denote summation over nearest neighbours in $U$; i.e., on indices $i,j\in \mathbb Z^d$ such that $\e i,\e j\in U$ and $|i-j|=1$.
If $U=\Omega$ then we simply write $X_\e(u)$. For such functions $u$ we define a piecewise-affine interpolation $A_\e(u)$ related to the scaling of a $1$-periodic subdivision of $\mathbb R^d$ into 
simplices whose restriction to the unit cube are $d!$ simplices, each of which has $d$ edges coinciding with the edges of the cube. The existence of such a decomposition, known as Kuhn decomposition, is ensured by the Sperner Lemma (see \cite{kuhn}). 
Note that these interpolations are well defined on each $U\subset\subset\Omega$ for $\e$ small. 

We have the following compactness result with respect to the convergence induced by the flat norm on the Jacobians of interpolations \cite[Theorem 3]{AC}.

\begin{lemma}[coerciveness of discrete energies] \label{lemma2}
Let $\{u^\e\}$ be a family of discrete functions such that $\sup_\e X_\e(u^\e)<+\infty$; then, there exist a subsequence $\e_j$ and an integral $(d-2)$-current $M$   
such that 
\begin{equation}\label{codico}
\lim_{j\to +\infty}{\bf F}_U(\star JA_{\e_j}(u^{\e_j})-\mu)=0 
\end{equation}
for all $U\subset\subset\Omega$ open sets, where $\mu=\pi M$. 
In particular, if $d=2$ there exist $N\in \mathbb N$, $x_\ell\in \Omega$, and $d_\ell\in\mathbb Z$ for $\ell\in\{1,\ldots,N\}$ such that $\mu$ is an atomic measure given by 
$$\mu=\pi\sum_{\ell=1}^N d_\ell \delta_{x_\ell}.$$ 
\end{lemma}

The previous lemma justifies a notion of convergence of discrete functions $u^\e$ to an integral $d-2$-current 
$M$, which will be sometimes used in the sequel, as the validity of
\begin{equation}\label{codico-2}
\lim_{\e\to 0}{\bf F}_U(\star JA_{\e}(u^{\e})-\mu)=0 
\end{equation}
 for all $U\subset\subset\Omega$ open sets, where $\mu=\pi M$.

With respect to this convergence we have the following result, which is a particular case of \cite[Theorem 3]{AC}.  
In the case $d=2$ this is also a particular case of \cite[Theorem 5]{AC}, obtained, in the notation therein, taking $c_\xi= 2$ if $|\xi|=1$ and $c_\xi= 0$ otherwise.

\begin{theorem}\label{theorem3}
The sequence $\{X_\e\}$ $\Gamma$-converges 
with respect to convergence \eqref{codico-2} to the functional $\Phi$ defined on integral $(d-2)$-currents as
$$\Phi(M)=4\pi \|M\|,$$ 
where $\|M\|$ is the mass of the current $M$. 
In particular, if $d=2$, then $M=\sum_{\ell=1}^N d_\ell \delta_{x_\ell}$ and   
$$\Phi(M)=4\pi \sum_{\ell=1}^N |d_\ell|.$$ 
\end{theorem}

\subsection{A notion of convergence for averages}\label{2comp} 
Let $\Omega\subset \mathbb R^d$ be a bounded open Lipschitz domain. For each $u\in L^1_{\rm loc}(\Omega; \mathbb R^2)$ and $\e>0$ we define a discretization $I_\e(u)\colon\e\mathbb Z^d\to\mathbb R^2$ as follows
\begin{equation}\label{def-ie}
I_\e(u)(\e i)=\frac{1}{\e^d}\int_{Q_\e^i\cap \Omega}u(x)\, dx,
\end{equation}
where $Q_\e^i=\e  i +\e[0,1]^d$. 

We consider  
a $1$-periodic triangulation of $\mathbb R^d$ with vertices in $\mathbb Z^d$  
such that its restriction to the unit cube is given by $d!$ simplices, each of which has $d$ edges coinciding with the edges of the cube (the Kuhn decomposition of the cube, see {\em e.g.} \cite[Lemma 1]{kuhn}). 
This triangulation corresponds to a family of piecewise-affine functions $\{\lambda_k\}_{k\in\mathbb Z^d}$ with $\lambda_k\colon\mathbb R^d\to[0,1]$ and such that $\lambda_k(k)=1$, the support of $\lambda_k$ is the union of the elements of the triangulation containing $k$ and $\sum_{k\in\mathbb Z^d}\lambda_k(x)=1$ for all $x$. 
Given a discrete function $v\colon\e \mathbb Z^d\to\mathbb R^d$, we define a piecewise-affine interpolation by setting
\begin{equation}\label{def-ae}
A_\e(v)(x)=\sum_{k\in\mathbb Z^d}\lambda_k\Big({x\over\e}\Big) v(\e k). 
\end{equation} 
Note that if the dimension $d$ is equal to $2$, then the Kuhn decomposition is simply the subdivision of each square in two triangles, which will make the proofs easier to follow in that case.   

\smallskip
We now introduce a notion of convergence of functions to integral currents. It involves the use the notion of Jacobian (of interpolations) for functions which in general are not weakly differentiable.
Its use will be justified by a compactness result (see  Theorem \ref{comp-convergence} (i) in the next section). 

\begin{definition}[Convergence]\label{conv}
Let $\{u_\e\}$ be a sequence in $L^2_{\rm loc}(\Omega;\mathbb R^2)$ and let $\mu=\pi M$, where 
$M$ is an integral $(d-2)$-current.  
The sequence $\{u_\e\}$ {\rm converges to} $\mu$ if 
for every $U\subset\subset\Omega$ we have  
$$\lim_{\e\to 0}{\bf F}_U(\star J(A_\e(I_\e(u_\e)))-\mu)=0,$$
where ${\bf F}_U$ is the flat norm in $U$ and $\star J$ is the Jacobian current. 
\end{definition}
\begin{remark}\label{conv-discreta}\rm 
If $v^\e\colon\e \mathbb Z^d\to S^1$ is such that the sum $X_\e(v^\e)$ defined in \eqref{XY} is equibounded, then Lemma \ref{lemma2} implies that, up to subsequences, there exists $\mu$ as above such that
$$\lim_{\e\to 0}{\bf F}_U(\star J(A_\e(v^\e)))-\mu)=0$$
for every $U\subset\subset\Omega$. 
In our case, we cannot directly use that result since the discrete functions $I_\e(u_\e)$ are not $S^1$-valued. 
\end{remark}

\subsection{The main result}
We are now in the position to state our main result of compactness and $\Gamma$-convergence, 
which will be proved in the following sections. 

\begin{theorem}[Compactness and $\Gamma$-convergence]\label{comp-convergence}
Let $\rho\colon[0,+\infty)\to[0,+\infty)$ be a fixed kernel with compact support such that 
\begin{enumerate}
    \item[\rm(a)] $\displaystyle\int_{\mathbb R^d}\rho(|\xi|)|\xi|^2\, dx<+\infty;$
    \item[\rm(b)] $\rho(t)\geq \rho_0 >0$ in a neighbourhood of $0$. 
\end{enumerate} 
Let $\Omega\subset\mathbb R^d$ be an open Lipschitz bounded domain, and for all $\e>0$ let $F_\e\colon L^1(\Omega; S^1)\to [0,+\infty]$ be defined by 
\begin{equation*}
F_\e(u)=
{1\over\e^{d+2}|\log\e|} \int_{\Omega\times\Omega}\rho\Bigl({|x-y|\over\e}\Bigr) {|u(x)-u(y)|^2}\dxy.
\end{equation*}
Then, we have the following results. 
\begin{enumerate}
\item[{\rm (i)}] {\rm (equi-coerciveness of $F_\e$)} 
If $\{u_\e\}$ is a sequence such that 
$F_\e(u_\e)$ is equibounded,  
then, up to subsequences, there exists an integral $(d-2)$-current $M$ 
such that $\{u_\e\}$ converges to $\mu=\pi M$ in the sense of Definition {\rm\ref{conv}}. 
\item[{\rm (ii)}] {\rm (lower bound)} 
If $\{u_\e\}$ converges to $\mu=\pi M$ in the sense of Definition {\rm\ref{conv}}, then 
\begin{equation}\label{Crho}
\liminf_{\e\to 0}F_\e(u_\e)\geq C_\rho \|M\|, \ \ \hbox{ where } \  C_\rho={2\pi\over d}\int_{\mathbb R^d}\rho(|\xi|)|\xi|^2\, d\xi.
\end{equation} 
\item[{\rm (iii)}] {\rm (upper bound)} For every integral $(d-2)$-current $\mu=\pi M$, there exists a sequence $\{u_\e\}$ converging to $\mu=\pi M$ in the sense of Definition {\rm\ref{conv}} such that 
$$\lim_{\e\to 0}F_\e(u_\e)= C_\rho \|M\|.$$
\end{enumerate}
\end{theorem}
 Items (ii) and (iii) in Theorem \ref{comp-convergence} state that the sequence $\{F_\e\}$ $\Gamma$-converges with respect to the convergence in Definition {\rm\ref{conv}} to the functional 
$$F(\mu)= C_\rho \|M\|, \ \ \hbox{ where } \  C_\rho={2\pi\over d}\int_{\mathbb R^d}\rho(|\xi|)|\xi|^2\, d\xi.$$
with domain 
$D(F)=\{\mu=\pi M:  M \ \hbox{\rm integral } (d-2)\hbox{\rm -current}\}$.

\section{The compactness result}\label{comp-sec}
In this section we prove item (i) of Theorem \ref{comp-convergence}; that is, the key compactness result with respect to the convergence of the Jacobian currents of interpolations of discrete averages of functions with bounded energies defined in Definition \ref{conv}. This is the analog of Lemma \ref{lemma2} in which instead the Jacobian currents of interpolations of discrete functions are taken into account. That lemma cannot directly be applied to discrete averages since the latter do not take values in $S^1$. We will use the fact that it can nevertheless be applied to discretizations defined on  translated lattices for almost all translations.

In order to prove the equi-coerciveness of the family $\{F_\e\}$ defined in \eqref{effe-e}, 
by scaling and comparison it is sufficient to treat the case 
\begin{equation}\label{effe-e-d-comp}
F_\e(u)={1\over\e^{d+2}|\log\e|} \int_{\Omega\times\Omega}\chi_{[-1,1]^d}\Bigl({x-y\over\e}\Bigr) {|u(x)-u(y)|^2}\dxy 
\end{equation}
with $u\in L^1(\Omega; S^1)$. 
For future reference, we restate the compactness result as follows. 
\begin{lemma}[Compactness]\label{compactness}
Let $\Omega\subset \mathbb R^d$ be a bounded open Lipschitz domain and let 
$\{u_\e\}$ be a sequence such that 
$F_\e(u_\e)$ is equibounded, where $F_\e$ is defined by
\eqref{effe-e-d-comp}. 
Then, up to subsequences, there exists an integral $(d-2)$-current $M$ 
such that $\{u_\e\}$ converges to $\mu=\pi M$ in the sense of Definition {\rm\ref{conv}}. 
\end{lemma}

\begin{proof}  
We first consider the two-dimensional case, in which a simplified notation makes proofs clearer. 

Let $u\in L^1(\Omega; S^1)$ and let $U\subset\subset\Omega$ be an open set. 
For $\e$ small enough, using a triangular argument we obtain 
\begin{eqnarray*}
&& \hskip-5mm{1\over\e^2|\log\e|}\int_{U} \Big|u(x+\frac{\e}{2} e_1)-u(x)\Big|^2\dx
\\
&&\le {1\over\e^2|\log\e|}\int_{[0,1]^2}\int_{U} \Big|u\Big(x+\frac{\e}{2} e_1\Big)-u\Big(x+\frac{\e}{2} e_1+\frac{\e}{2}\xi\Big)\\
&&\hskip5cm+u\Big(x+\frac{\e}{2} e_1+\frac{\e}{2}\xi\Big)-u(x)\Big|^2\dx\, d\xi\\
&&\le {2\over\e^2|\log\e|}\int_{[0,1]^2}\int_{U}\Bigl( \Big|u\Big(x+\frac{\e}{2} e_1\Big)-u\Big(x+\frac{\e}{2} e_1+\frac{\e}{2}\xi\Big)\Big|^2\\
&&\hskip5cm+\Big|u\Big(x+\frac{\e}{2} e_1+\frac{\e}{2} \xi\Big)-u(x)\Big|^2\Bigr)
\dx\, d\xi
\\
&&\le {32\over\e^2|\log\e|}\int_{[0,1]^2}\int_{U+\e [0,1]^2} |u(x+\e\eta)-u(x)|^2\dx\, d\eta\\
&&\le 
{32\over\e^4|\log\e|} \int_{\Omega\times\Omega}\chi_{[-1,1]^2}\Big(\frac{x-y}{\e}\Big) |u(x)-u(y)|^2\dx\, dy\\
&&= 
32 F_\e(u),
\end{eqnarray*} 
and hence for all $u\in L^1(\Omega; S^1)$ and $U\subset\subset \Omega$ we get 
\begin{eqnarray}\label{e1}
&&\hspace{-5mm}{1\over\e^2|\log\e|}\int_{U} |u(x+\e e_1)-u(x)|^2\dx\nonumber \\
&&\leq 
{2\over\e^2|\log\e|}\int_{U} \Big|u(x+\e e_1)-u\Big(x+\frac{\e}{2} e_1\Big)\Big|^2\dx+
{2\over\e^2|\log\e|}\int_{U} \Big|u\Big(x+\frac{\e}{2} e_1\Big)-u(x)\Big|^2\dx\nonumber \\
&&\leq C\, F_\e(u)
\end{eqnarray} 
for $\e$ small enough.

Let $\{u_\e\}$ be such that $\sup_\e F_\e(u_\e)<+\infty$. 
For $U\subset\subset\Omega$ and $\e>0$,   
we set 
$$\mathcal I_\e(U)=\mathcal I_\e=\{k\in\mathbb Z^2: \e k+[-2\e,2\e]^2\subset\subset \Omega, \ 
\e k+[-\e,\e]^2\cap U\neq\emptyset\}.$$
Upon using in \eqref{e1} any set containing $\{x\in \Omega: \hbox{\rm dist}(x,U)\leq 2\sqrt 2\e\}$, 
which we may suppose to be compactly contained in $\Omega$, 
 in the place of $U$, 
we get the estimate 
\begin{eqnarray}\label{stimae1}
C\, F_\e(u_\e)&\ge&{1\over\e^2|\log\e|}
\sum_{k\in\mathcal I_\e} 
\int_{\e (k+ [0,1]^2)} |u_\e(x+\e e_1)-u_\e(x)|^2\dx\nonumber\\
&=&\frac{1}{|\log\e|}
\sum_{k\in\mathcal I_\e} 
\int_{[0,1]^2} |u_\e(\e k+\e z+\e e_1)-u_\e(\e k+\e z)|^2\, dz\nonumber\\
&=&
\int_{[0,1]^2}\frac{1}{|\log\e|}\sum_{k\in\mathcal I_\e} |u_\e(\e k+\e z+\e e_1)-u_\e(\e k+\e z)|^2\, dz 
\end{eqnarray}
for $\e$ small enough. 
The same argument can be also used with $e_2$ in the place of $e_1$. 
Now, we define a family of discrete functions 
$u^{\e,z}\colon\e\mathbb Z^2\to S^1$ given by  $u^{\e,z}_k=u^{\e,z}(\e k)=u_\e(\e k+\e z)$, after extending $u_\e$ to $\mathbb R^2$ by setting $u_\e=0$ outside $\Omega$ for convenience. 
It follows that  
\begin{equation}\label{stimaxe}
\int_{[0,1]^2} X_\e(u^{\e,z};U)\,dz\le C\,F_\e(u_\e),  
\end{equation} 
where for $v\colon \e\mathcal I_\e\to S^1$ we have set 
$$X_\e(v; U)=\frac{1}{|\log\e|}\sum_{\substack{\langle i,j\rangle \\ i,j\in\mathcal I_\e}}|v(\e i)-v(\e j)|^2.$$

We first outline the arguments of the rest of the proof. 
In order to prove the claim we 
will use Lemma \ref{lemma2} to obtain that, up to subsequences, the sequence of piecewise-affine functions $\{A_\e(u^{\e,z})\}$ converges to a measure $\mu^z$ of the form $$\mu^z=\pi\sum_{\ell=1}^{N_z}d_\ell^z\delta_{x^z_\ell}$$ in the sense that \begin{equation}\label{convuze}
\lim_{\e\to 0}{\bf F}_U(\star J(A_\e(u^{\e,z}))-\mu^z)=0,
\end{equation}
and at the same time (in order to have a common subsequence) we apply Lemma \ref{lemma1} to deduce that, up to subsequences, the sequence $\{u_\e\}$ converges to one of such $\mu^z$. Note that as a consequence we prove that such $\mu^z$ are in fact almost all independent of $z$. 

\medskip
As a first step, in order to compare $u^{\e,z}$ and $A_\e(I_\e (u_\e))$ we introduce the following auxiliary functions. For any $z\in [0,1]^2$ we consider the piecewise-affine function which interpolates on the values of $u_\e$ on the lattice $\e z+\e \mathbb Z^2$; that is, the values of $u^{\e,z}$. We set 
$$u_\e^z(x)=A_\e(u^{\e,z})(x-\e z)=\sum_{k\in\mathbb Z^2}\lambda_k\Big(\frac{x-\e z}{\e}\Big) u_\e(\e (k+z))=\sum_{k\in\mathbb Z^2}\lambda_k\Big(\frac{x-\e z}{\e}\Big) u_k^{\e,z}.$$ 
Note that if \eqref{convuze} holds then $\{u_\e^z\}$ still converges to the same $\mu^z$,  in the sense that \begin{equation}\label{convuze-2}
\lim_{\e\to 0}{\bf F}_U(\star J(u^z_\e)-\mu^z)=0.
\end{equation}

In order to apply Lemma \ref{lemma1}, we have to estimate 
the $L^2$-norm of $A_\e(I_\e(u_\e))-u^z_\e$ and of its gradient. 
To this end, we note that  
\begin{equation}\label{graduz}
|\nabla u_\e^z(x)|^2=
\begin{cases}
\vspace{2mm}\displaystyle \Big|\frac{u_{k+e_1}^{\e,z}-u_{k}^{\e,z}}{\e}\Big|^2+
\Big|\frac{u_{k+e_2}^{\e,z}-u_{k}^{\e,z}}{\e}\Big|^2 &\hbox{ for }\ x\in\e z+\e k+ \e Q^-\\
\displaystyle \Big|\frac{u_{k+e_1+e_2}^{\e,z}-u_{k+e_1}^{\e,z}}{\e}\Big|^2+
\Big|\frac{u_{k+e_1+e_2}^{\e,z}-u_{k+e_2}^{\e,z}}{\e}\Big|^2 &\hbox{ for }\ x\in\e z+ \e k+ \e Q^+, 
\end{cases}
\end{equation} 
where $k\in\mathcal I_\e$ we have set $Q^-=\{(x_1,x_2)\in[0,1]^2: x_1+x_2\leq 1\}$ and $Q^+=[0,1]^2\setminus Q^-$. 
By \eqref{stimae1} and \eqref{graduz} we get 
\begin{eqnarray*}
&&\hspace{-2cm}\int_{[0,1]^2}\frac{1}{|\log\e|}\int_U |\nabla u_\e^z(x)|^2\, dx\, dz\\ 
&\le&
\int_{[0,1]^2}\frac{2}{|\log\e|}\sum_{k\in\mathcal I_\e} \e^2\Big(\Big|\frac{u_\e(\e k+\e z+\e e_1)-u_\e(\e k+\e z)}{\e}\Big|^2 \\
&&\hskip3.1cm+\Big|\frac{u_\e(\e k+\e z+\e e_2)-u_\e(\e k+\e z)}{\e}\Big|^2\Big)\, dz\\
&\le& CF_\e(u_\e).
\end{eqnarray*}
Similarly, noting that by Jensen's inequality we have 
$$
\int_{\e (k+ [0,1]^2)} |u_\e(x+\e e_1)-u_\e(x)|^2\dx\geq \e^2|(I_\e(u_\e))(\e k+\e e_1)-(I_\e(u_\e))(\e k)|^2$$ 
for $k\in\mathcal I_\e$, we deduce that   
\begin{equation}\label{conf-Aeu-g}
\frac{1}{|\log\e|}
\int_{U} |\nabla A_\e(I_\e(u_\e))(x)|^2\, dx\le C F_\e(u_\e). 
\end{equation} 
In particular it follows that 
\begin{equation}\label{stimagrad}
\frac{1}{|\log\e|}\int_{[0,1]^2}
\int_{U} |\nabla A_\e(I_\e(u_\e))-\nabla u_\e^z|^2\, dx\, dz\le C\,F_\e(u_\e). 
\end{equation}

As for the $L^2$-norm of $A_\e(I_\e(u_\e))-u_\e^z$, we start by estimating  $u_\e^z-u_\e$. We have 
\begin{eqnarray}\label{stima-z}
&&\hspace{-1cm}\int_{[0,1]^2}\int_{U}|u_\e^z(x)-u_\e(x)|^2\, dx\, dz\nonumber \\
&&=\int_{[0,1]^2}\int_{U}\Big|\sum_{k\in\mathbb Z^2}\lambda_k\Big(\frac{x-\e z}{\e}\Big) \big(u_\e(\e (k+z))-u_\e(x)\big)\Big|^2\, dx\, dz\nonumber\\
&&\leq \int_{[0,1]^2}\int_{U}\sum_{k\in\mathbb Z^2}\lambda_k\Big(\frac{x-\e z}{\e}\Big) \Big|u_\e(\e (k+z))-u_\e(x)\Big|^2\, dx\, dz\nonumber\\
&&=\int_{U} \int_{[0,1]^2}\sum_{k\in\mathbb Z^2}\lambda_k\Big(\frac{x-\e z}{\e}\Big) \Big|u_\e(\e (k+z))-u_\e(x)\Big|^2\, dz\, dx\nonumber\\
&&\leq\sum_{j\in\mathcal I_\e}\int_{Q_\e^j} \int_{[0,1]^2}\sum_{k\in\mathbb Z^2}\lambda_k\Big(\frac{x-\e z}{\e}\Big) \Big|u_\e(\e (k+z))-u_\e(x)\Big|^2\, dz\, dx\nonumber\\
&&\leq\sum_{j\in\mathcal I_\e}\int_{Q_\e^j} \frac{C}{\e^2}\int_{x+[-\e,\e]^2} |u_\e(y)-u_\e(x)|^2\, dy\, dx\nonumber\\
&&\leq\frac{C}{\e^2}\int_{\Omega\times\Omega}\chi_{[-1,1]^2}\Big(\frac{x-y}{\e}\Big) |u_\e(y)-u_\e(x)|^2\, dy\, dx\nonumber\\
&& 
=C\e^2|\log\e| F_\e(u_\e). 
\end{eqnarray} 
In this inequality, we used the fact that only an equibounded number of $\lambda_k$ is different from $0$ at fixed $x$ and that $\lambda_k\leq 1$. 

Recalling \eqref{def-ie} and \eqref{def-ae}, we can write 
\begin{eqnarray*}
A_\e(I_\e(u_\e))(x)
&=&\sum_{k\in\mathbb Z^2}\lambda_k\big(\frac{x}{\e}\big)\frac{1}{\e^2}\int_{Q_\e^k}u_\e(z)\, dz\\
&=&\int_{[0,1]^2}\sum_{k\in\mathbb Z^2}\lambda_k\big(\frac{x}{\e}\big)u_\e(\e z+\e k)\, dz\ =\ \int_{[0,1]^2}u_\e^z(x+\e z)\, dz. 
\end{eqnarray*}
Hence, proceeding with estimates as in \eqref{stima-z}, 
\begin{eqnarray}\label{conf-Aeu}\nonumber
\int_U| A_\e(I_\e(u_\e))(x) - u_\e(x)|^2\, dx
&=&\int_U\Big|\int_{[0,1]^2} (u_\e^z (x+\e z) - u_\e(x))\, dz\Big|^2\, dx\\\nonumber
&\leq&\int_U\int_{[0,1]^2} |u_\e^z (x+\e z) - u_\e(x)|^2\, dz\, dx\\
&\leq& C\e^2|\log \e| F_\e(u_\e).
\end{eqnarray}
In particular, this estimate and \eqref{stima-z} 
imply that 
\begin{equation}\label{stimaL2}
\frac{1}{\e^2|\log\e|}\int_{[0,1]^2}\int_U| A_\e(I_\e(u_\e (x))) - u^z_\e(x)|^2\, dx\, dz\le C F_\e(u_\e).
\end{equation}

By estimates \eqref{stimaxe}, \eqref{stimagrad} and \eqref{stimaL2}, applying Fatou's Lemma, we deduce that 
for almost all $z\in[0,1]^2$ 
\begin{eqnarray}\label{3}
 &&\liminf_{\e\to 0}\Big(X_\e(u^{\e,z};U)+\frac{1}{\e^2|\log\e|}\int_U| A_\e(I_\e(u_\e (x))) - u^z_\e(x)|^2\, dx\nonumber\\
 &&\hspace{3.7cm}+
 \frac{1}{|\log\e|}\int_U| \nabla A_\e(I_\e(u_\e)) - \nabla u^z_\e(x)|^2\, dx\Big)<+\infty.
\end{eqnarray}
Recalling Lemma \ref{lemma2}, 
for almost all $z\in[0,1]^2$ there exists a subsequence $\e_j=\e_j(z)$ and a measure $\mu^z$ such that $u^{\e_j,z}\to\mu^z$ and 
$$\frac{1}{\e_j^2|\log\e_j|}\int_U| A_\e(I_\e(u_{\e_j} (x))) - u^z_{\e_j}(x)|^2\, dx+
 \frac{1}{|\log\e_j|}\int_U| \nabla A_{\e_j}(I_{\e_j}(u_{\e_j})) - \nabla u^z_{\e_j}(x)|^2\, dx\le C.$$
By Lemma \ref{lemma1} we obtain that 
\begin{equation}\label{muz}
\lim_{j\to +\infty}{\bf F}_U(\star J(A_{\e_j}(I_{\e_j}(u_{\e_j})))-\mu^z)=0;
\end{equation}
that is, $\{u_{\e_j}\}$ converges to $\mu^z$ in the sense of Definition \ref{conv}.

\bigskip 

The proof in the general case $\Omega\subset \mathbb R^d$ is exactly the same as in the $2$-dimensional case, with the difference that all triangular arguments must be repeated $d$ times, a factor $\e^{d-2}$ appears in the definition of the discrete functional $X_\e$ and we have to use the properties of the Kuhn triangulation in the computation of the gradient of $u_\e^z$, generalizing \eqref{graduz}. We omit the details since they only result in  a heavier notation.  
\end{proof}

\begin{remark}[Independence from the discretization]\label{indepp}\rm 
We now remark that Definition \ref{conv} is in fact independent from the choice of the discretization. 
We make this statement precise in the case $d=2$, the general case following with minor modifications. 

For a fixed $\xi=(\xi_1,\xi_2)\in\mathbb R^2\setminus\{0\}$
we consider the lattice $\mathbb Z_\xi^2=\mathbb Z \xi \oplus \mathbb Z \xi^\perp$, 
where $\xi^\perp=(\xi_2,-\xi_1)$ and the convergence of a sequence $\{u_\e\}$ obtained by a discretization on this lattice.   
Namely, for each $u\in L^1_{\rm loc}(\Omega; \mathbb R^2)$ and $\e>0$ we define $I^\xi_\e(u)\colon\e\mathbb Z^2_\xi \to\mathbb R^2$ as follows
\begin{equation}
I_\e^\xi(u)(\e k)=
\frac{1}{|\xi|^2\e^2}\int_{Q_\e^{k,\xi}\cap \Omega}u(x)\, dx,
\end{equation}
where $Q_\e^{k,\xi}=\e k+ \e Q^\xi$, $k\in \mathbb Z^2_\xi$ and $Q^\xi=[0,1]\xi+[0,1]\xi^\perp$.    

Then, we consider a fixed triangulation of $\mathbb R^2$ with vertices in $\mathbb Z^2_\xi$ and the corresponding family of piecewise-affine functions $\{\lambda_k^\xi\}_{k\in\mathbb Z^2_\xi}$ with $\lambda_k^\xi\colon\mathbb R^2\to[0,1]$ and such that $\lambda_k^\xi(k)=1$, the support of $\lambda_k^\xi$ is the union of the elements of the triangulation containing $k$ and $\sum_{k\in\mathbb Z^2_\xi}\lambda_k(x)=1$ for all $x$. 
Given a discrete function $v\colon\e \mathbb Z^2_\xi\to\mathbb R^2$, following \eqref{def-ae} we define a piecewise-affine interpolation by setting
\begin{equation}
A_\e^\xi(v)(x)=\sum_{k\in\mathbb Z^2_\xi}\lambda_k^\xi\Big({x\over\e}\Big) v(\e k). 
\end{equation}

The key argument of Lemma \ref{compactness}  is the comparison of the functions $u_\e^z$ with $u_\e$ as in \eqref{stima-z}, and is obtained thanks to \eqref{e1}. 
We can repeat the arguments leading to \eqref{e1} with $\xi$ in the place of $e_1$, up to changing the constants, as follows. 
For fixed $\xi$ and $z\in[0,1]^2$, for $i\in\mathbb Z^2$ we set 
\begin{equation}\label{uexiz}
u^{\e,z,\xi}_i=u^{\e,z,\xi}(\e i)=u_\e(\e z_1\xi+\e z_2\xi^\perp+\e i_1 \xi +\e i_2 \xi^\perp), 
\end{equation} 
and let $u^{z,\xi}_\e$ denote the corresponding piecewise-affine interpolation from the lattice $\e z_1\xi+\e z_2\xi^\perp+\e \mathbb Z^2_\xi$. 
Note that $u^{\e,z,\xi}_i$ are defined for $i\in\mathbb Z^2$, so that they differ form $u_\e^{z,\xi}$ by the linear transformation carrying $\mathbb Z^2$ in $\mathbb Z^2_\xi$, up to a small translation.  
Then we obtain the estimates analogous to \eqref{conf-Aeu-g} and \eqref{conf-Aeu} with $A^\xi_{\e}(I^\xi_{\e}(u_{\e}))$ in the place of $A_{\e}(I_{\e}(u_{\e}))$, and we can conclude that 
\begin{eqnarray}\nonumber&&\hskip-1.5cm\int_{[0,1]^2} X_{\e} (u^{\e,z,\xi};U^\e_{z,\xi})\,dz +\frac{1}{\e^2|\log\e|}\int_{[0,1]^2} \int_U| A^\xi_{\e}(I^\xi_{\e}(u_{\e} (x))) - u^{z,\xi}_{\e}(x)|^2\, dx\,dz\\
&&\hskip1.5cm+
 \frac{1}{|\log{\e}|}\int_{[0,1]^2} \int_U| \nabla A^\xi_{\e}(I^\xi_{\e}(u_{\e})) - \nabla u^{z,\xi}_{\e}(x)|^2\, dx\,dz\le C(|\xi|) F_\e(u_\e), 
 \end{eqnarray}
where
$U^\e_{z,\xi}=\{x\in\mathbb R^2: \e z_1 \xi+\e z_2\xi^\perp +L(\xi)x\in U\}$ 
and $L(\xi)$ is the linear map such that $L(\xi)(e_1)=\xi$ and $L(\xi)(e_2)=\xi^\perp$. 

As a result, concluding as in the proof of the proposition, we obtain 
a subsequence $\e_j=\e_j(\xi)$ and 
a measure 
$\mu(\xi)$ such that 
\begin{equation}\label{defxi}\lim_{j\to+\infty}{\bf F}_U(\star J(A^\xi_{\e_j}(I_{\e_j}^\xi(u_{\e_j})))-\mu(\xi))=0\end{equation} 
for all $U\subset\subset\Omega$. 
Note that if $\mu(\xi)=\pi\sum_{\ell=1}^Nd^\xi_\ell\delta_{x^\xi_\ell}$, then $u^{\e_j,z,\xi}$ converges to 
$\pi\sum_{\ell=1}^Nd^\xi_\ell\delta_{L^{-1}(x^\xi_\ell)}$. 

Now, we can apply Lemma \ref{lemma2} to $A^\xi_{\e_j}(I_{\e_j}^\xi(u_{\e_j}))$ and $A_{\e_j}(I_{\e_j}(u_{\e_j}))$ with the estimate on the gradients given by \eqref{conf-Aeu-g}, while the second assumption can be obtained from \eqref{conf-Aeu} by a triangular argument. 
Hence, the sequence $\{u_\e\}$ converges up to subsequences to the measure $\mu(\xi)$ in the sense of Definition \ref{conv},  
showing that the convergence indeed depends only on $u_\e$ and not on the discretization chosen.
Indeed, if the sequence $u_\e$ converges to $\mu$ in the sense of Definition \ref{conv}, 
then $u_{\e_j}$ converges to $\mu$, implying $\mu=\mu(\xi)$. 
This allows to remark that, if in Definition \ref{conv} we require that \eqref{defxi} holds 
for all $U\subset\subset\Omega$, 
we obtain a definition of convergence which is equivalent to Definition \ref{conv}. 
\end{remark}

\goodbreak
\section{Proof of the $\Gamma$-convergence result}\label{conv-sec} 
We can now prove the $\Gamma$-convergence result. We recall that 
$\rho\colon[0,+\infty)\to[0,+\infty)$ is a fixed kernel with compact support such that 
$\int_{\mathbb R^d}\rho(|\xi|)|\xi|^2\, dx<+\infty$. 
Moreover, we suppose that $\rho$ is larger than a strictly positive constant in a neighbourhood of $0$; then, upon scaling, it is not restrictive to suppose that \begin{equation}\label{ipotesi-rho}
\rho\geq \rho_0\chi_{[-1,1]^d} ,
\end{equation}
so that we can directly apply the compactness result of Lemma \ref{compactness}.  
We also recall that $\Omega\subset\mathbb R^d$ is an open Lipschitz bounded domain, and  $F_\e\colon L^1(\Omega; S^1)\to [0,+\infty]$ is defined by 
\begin{equation*}
F_\e(u)=
{1\over\e^{d+2}|\log\e|} \int_{\Omega\times\Omega}\rho\Bigl({|x-y|\over\e}\Bigr) {|u(x)-u(y)|^2}\dxy.
\end{equation*}

The proof of the lower bound uses a discretization approach, rewriting the energies as integrals on $\mathbb R^d\times [0,1]^d$ of lattice energies parameterized on the integration variable, which can be analyzed separately, while a direct computation is used for the upper bound. 
 
\begin{proof}[Proof of Theorem {\rm\ref{comp-convergence}} {\rm (ii)} (lower bound)] 
In order to point out and clarify the key steps of the proof, we first deal with the case $d=2$. 
\smallskip 

\noindent(a) {\em The case $d=2$.} \ In this case, we can consider an integral $0$-current of the form 
$M=\sum_{\ell=1}^nd_{\ell}\delta_{x_\ell}$. 
Let $u_\e$ converge to $\mu=\pi M$ in the sense of Definition 
\ref{conv}. 
Let $U\subset\subset\Omega$. In analogy with the definition of the family of indices $\mathcal I_\e$, we set 
$$\mathcal I_\e^\xi(U)=\mathcal I_\e^\xi=\{k\in\mathbb Z^2_\xi: \e k+4\e \widehat Q^\xi\subset\subset \Omega, \ 
\e k+2\e \widehat Q^\xi\cap U\neq\emptyset\},$$ 
where $\widehat Q^\xi$ is the square centered at $0$ with edges $\xi$ and $\xi^\perp$. 
We have 
\begin{eqnarray}\label{primainf}
F_\e(u_\e)&=&\int_{\mathbb R^2}\rho(|\xi|) \Bigl({1\over\e^2|\log\e|}\int_{\Omega} |u_\e(x+\e\xi)-u_\e(x)|^2\dx\Bigr)\,d\xi\nonumber\\
&\geq& \frac{1}{2|\log\e|}\int_{\mathbb R^2}\rho(|\xi|)|\xi|^2 
\Bigl(\int_{[0,1]^2}\sum_{k\in\mathcal I_\e^\xi} |u_\e(\e s_1 \xi+ \e s_2\xi^\perp+\e k+\e \xi)\nonumber\\
&&\hspace{5cm}-u_\e(\e s_1 \xi+\e s_2\xi^\perp +\e k)|^2 \, ds\Bigr)\,d\xi\nonumber\\
&&+\frac{1}{2|\log\e|}\int_{\mathbb R^2}\rho(|\xi|)|\xi|^2 
\Bigl(\int_{[0,1]^2}\sum_{k\in\mathcal I_\e^\xi} |u_\e(\e s_1 \xi+\e s_2\xi^\perp+\e k+\e \xi^\perp)\nonumber\\
&&\hspace{5cm}-u_\e(\e s_1 \xi+\e s_2\xi^\perp +\e k)|   \, ds\Bigr)\,d\xi\nonumber\\
&\geq& \frac{1}{2|\log\e|}\int_{\mathbb R^2}\rho(|\xi|)|\xi|^2 
\Bigl(\int_{[0,1]^2}\frac{1}{2}\sum_{\langle i, j\rangle} |u_i^{\e,s,\xi}-u_j^{\e,s,\xi}|^2 \, ds\Bigr)\,d\xi\nonumber\\
&\geq& \frac{1}{4}\int_{\mathbb R^2}\rho(|\xi|)|\xi|^2 
\Bigl(\int_{[0,1]^2}X_\e (u^{\e,s,\xi}; U^\e_{s,\xi}) \, ds\Bigr)\,d\xi
\end{eqnarray}
where for $i\in\mathbb Z^2$ the value $u^{\e,s,\xi}_i$ is defined as in \eqref{uexiz} and  
$$U^\e_{s,\xi}=\{x\in\mathbb R^2: \e s_1 \xi+\e s_2\xi^\perp +L(\xi)(x)\in U\},$$  
 $L(\xi)$ being the linear map such that $L(\xi)(e_1)=\xi$ and $L(\xi)(e_2)=\xi^\perp$. 

We now introduce a parameter $\eta>0$ and write
\begin{eqnarray*}
F_\e(u_\e)& =&
(1-\eta)F_\e(u_\e) +\eta F_\e(u_\e)\\
&\ge& \frac{1}{2}\int_{\mathbb R^2\setminus B_\eta}\rho(|\xi|)|\xi|^2 
\int_{[0,1]^2}(1-\eta)\Bigl(X_\e (u^{\e,s,\xi};U_\xi) \\
&&\hskip3cm +C_\eta\frac{1}{\e^2|\log\e|}\int_U| A_\e(I_\e(u_\e (x))) - u^{s,\xi}_\e(x)|^2\, dx\nonumber\\
 &&\hskip3cm +C_\eta
 \frac{1}{|\log\e|}\int_U| \nabla A_\e(I_\e(u_\e)) - \nabla u^{s,\xi}_\e(x)|^2\, dx\Big)\, ds\,d\xi,
\end{eqnarray*}
where $U_\xi$ is any open set contained in the intersection of the sets $U^\e_{s,\xi}$ for $s\in[0,1]^2$ and $\e$ small enough, $u^{s,\xi}_\e$ denotes the corresponding piecewise-affine interpolation from the lattice $\e s_1\xi+\e s_2\xi^\perp+\e \mathbb Z^2_\xi$ as in Remark \ref{indepp} and $C_\eta$ is a positive constant.

By applying Fatou's Lemma we note that for almost all $\xi$ and $s$ the limit 
\begin{eqnarray*}
&&L(\xi,s)=\liminf_{\e\to 0} \Bigl(X_\e (u^{\e,s,\xi};U_\xi) +C_\eta\frac{1}{\e^2|\log\e|}\int_U| A_\e(I_\e(u_\e (x))) - u^{s,\xi}_\e(x)|^2\, dx\nonumber\\
 &&\hskip5cm +C_\eta
 \frac{1}{|\log\e|}\int_U| \nabla A_\e(I_\e(u_\e)) - \nabla u^{s,\xi}_\e(x)|^2\, dx\Big)
\end{eqnarray*}
is finite. Hence we can find a sequence $\e_j$ (depending on $\xi$ and $s$) such that 
\begin{eqnarray*}
&&L(\xi,s)=\lim_{j\to+\infty} \Bigl(X_{\e_j} (u^{\e_j,s,\xi};U_\xi)+C_\eta\frac{1}{\e_j^2|\log\e_j|}\int_U| A_{\e_j}(I_{\e_j}(u_{\e_j} (x))) - u^{s,\xi}_{\e_j}(x)|^2\, dx\\&&\hskip5cm +C_\eta
 \frac{1}{|\log{\e_j}|}\int_U| \nabla A_{\e_j}(I_{\e_j}(u_{\e_j})) - \nabla u^{s,\xi}_{\e_j}(x)|^2\, dx\Big)
\end{eqnarray*}
and equals the liminf above.  
By Remark \ref{indepp} we deduce that $u^{\e_j,s,\xi}$ converges to 
$$\hat\mu(\xi)=\pi\sum_{\ell=1}^Nd_\ell\delta_{L^{-1}(\xi)(x_\ell)},$$ 
so that
$$L(\xi,s)\ge \liminf_{j\to+\infty}X_{\e_j} (u^{\e_j,s,\xi};U_\xi)\ge 4\pi\sum_{\{\ell: x_\ell\in U\}} |d_\ell|$$ 
since we may assume that $L^{-1}(\xi)(x_\ell)\in U_\xi$. 

We can then proceed in the application of Fatou's Lemma to deduce that
$$\liminf_{\e\to0}F_\e(u_\e)\ge (1-\eta)\int_{\mathbb R^2\setminus B_\eta}\rho(|\xi|)|\xi|^2d\xi\,  \pi\sum_{\{\ell: x_\ell\in U\}}|d_\ell|,
$$
and finally, using the arbitrariness of $\eta$ and $U$, that
$$\liminf_{\e\to0}F_\e(u_\e)\ge \int_{\mathbb R^2}\rho(|\xi|)|\xi|^2d\xi\,  \pi\sum_{\ell=1}^N|d_\ell|,$$
which is the desired lower bound.

\smallskip 

\noindent(b) {\em The general case.}\ 
In the $d$-dimensional case we cannot simply take into account orthogonal bases of the form $\{\xi, \xi^\perp\}$. 
In order to repeat the argument in the $2$-dimensional case using the lower estimate for the discrete functionals $X_\e$, we will consider the space of the orthonormal bases in $\mathbb R^d$ 
and use them to parameterize the interpolations, generalizing the role of $\xi$ and $\xi^\perp$ in the $2$-dimensional computations.  
To this end, we define the space
$$V=\{\overline \nu=(\nu_1,\dots, \nu_d): \nu_j\in S^{d-1}  \ \hbox{\rm such that }  \langle \nu_i, \nu_j\rangle=0 \ \hbox{\rm for } i\neq j\},$$ 
whose Hausdorff dimension $\frac{d(d-1)}{2}$ is denoted by $k_d$. 
Moreover, for any $\nu\in S^{d-1}$ and $n\in \{1,\dots, d\}$ we set 
$$V^{\nu}_n=\{\overline \nu=(\nu_1,\dots, \nu_d)\in V: \nu_n=\nu\}.$$ 
The Hausdorff dimension of $V^{\nu}_n$ is $k_d-(d-1)$ and we have 
\begin{equation}\label{dim} 
\mathcal H^{k_d-(d-1)}(V^{\nu}_n)=\frac{\mathcal H^{k_d}(V)}{\mathcal H^{d-1}(S^{d-1})}.
\end{equation} 
Note that in fact the space $V$ is the orthogonal group $O(d)\subset GL(d)$; that is, the group of the orthogonal $d\times d$ matrices, and each $V^{\nu}_n$ corresponds to the orthogonal group $O(d-1)$ acting on the orthogonal complement of $\nu$. 

Now, let $M$ be an integral $(d-2)$-current, and let the sequence $\{u_\e\}$ converge to $\mu=\pi M$ in the sense of Definition \ref{conv}. 
Let $U\subset\subset\Omega$. In analogy with the previous case, for $t>0$ and $\overline \nu\in V$ we set 
$$\mathcal I_\e^{t\overline \nu}(U)=\mathcal I_\e^{t\overline \nu}=\{k\in\mathbb Z^d_{t\overline \nu}: \e k+4\e \widehat Q^{t\overline \nu}\subset\subset \Omega, \ 
\e k+2\e \widehat Q^{t\overline \nu}\cap U\neq\emptyset\},$$ 
where $\mathbb Z^d_{t\overline \nu}=
\mathbb Z t\nu_1 \oplus \dots \oplus \mathbb Z t\nu_1$ and $\widehat Q^{t\overline \nu}$ is the square centered at $0$ with edges $\{t\nu_1\}_{n=1}^{d}$. Again note that $\{t\nu_1, \dots, t\nu_d\}$ plays the same role as $\{\xi, \xi^\perp\}$ in the $2$-dimensional case. 
As in \eqref{primainf}, by using \eqref{dim} we have 
\begin{eqnarray}\label{stimainfd}
F_\e(u_\e)&=&\int_{\mathbb R^d}\rho(|\xi|) \Bigl({1\over\e^2|\log\e|}\int_{\Omega} |u_\e(x+\e\xi)-u_\e(x)|^2\dx\Bigr)\,d\xi\nonumber\\
&=&\int_0^{+\infty} \rho(t) t^{d-1}\int_{S^{d-1}} \frac{1}{\e^2|\log\e|} 
\int_{\Omega} |u_\e(x+\e t\nu)-u_\e(x)|^2\dx \, d\mathcal H^{d-1}(\nu)\, dt
\nonumber\\
&=&\int_0^{+\infty} \rho(t) t^{d-1}\int_{S^{d-1}} \frac{1}{d}\sum_{n=1}^d \frac{\mathcal H^{d-1}(S^{d-1})}{\mathcal H^{k_d}(V)}\int_{V^{\nu}_n} \frac{1}{\e^2|\log\e|}  \nonumber 
\\ 
&&{\hspace{1cm}}
\int_{\Omega} |u_\e(x+\e t\nu)-u_\e(x)|^2\dx \, d\mathcal H^{k_d-(d-1)}(\overline \nu_n)\, d\mathcal H^{d-1}(\nu)\, dt\nonumber \\
&=&\frac{1}{d}\frac{\mathcal H^{d-1}(S^{d-1})}{\mathcal H^{k_d}(V)} \int_0^{+\infty} \rho(t) t^{d-1} \int_{V} \frac{1}{\e^2|\log\e|}  
\nonumber 
\\ 
&&{\hspace{1cm}}
\int_{\Omega}\sum_{n=1}^d |u_\e(x+\e t\nu_n)-u_\e(x)|^2\dx \, d\mathcal H^{k_d}(\overline \nu)\, dt,  
\end{eqnarray}
where $\overline \nu_n=(\nu_1, \dots, \nu_{n-1},\nu_{n+1},\dots, \nu_d)\in (S^{d-1})^{d-1}$ parameterizes an element of 
$V^{\nu}_n$ by not considering the $n$-th component $\nu$. 
Now, proceeding as in the $2$-dimensional case, we subdivide $\Omega$ in cubes corresponding to indices in $\mathcal I_\e^{t\overline \nu}$, obtaining 
\begin{eqnarray*}
&&\hspace{-5mm}\frac{1}{\e^2|\log\e|}  \int_{\Omega}\sum_{n=1}^d |u_\e(x+\e t\nu_n)-u_\e(x)|^2\dx \\
&&\geq t^d \int_{[0,1]^d} \frac{1}{|\log\e|}\sum_{k\in \mathcal I_\e^{t\overline \nu}} \e^{d-2} 
\Big|u_\e\Big(\e \sum_{\ell=1}^d t s_\ell\nu_\ell +\e t\nu_n +\e k\Big)-u_\e\Big(\e \sum_{\ell=1}^d t s_\ell\nu_\ell +\e k\Big)\Big|^2\, ds\\
&&=t^d \int_{[0,1]^d} \frac{1}{|\log\e|} \frac{1}{2}\sum_{\langle i, j \rangle} |u_i^{\e, s, t\overline \nu}-u_j^{\e, s, t\overline \nu}|^2 \, ds\\
&&\geq \frac{t^d}{2} \int_{[0,1]^d} X_\e(u^{\e,s, t\overline \nu}; U^\e_{s,t\overline\nu}) \, ds, 
\end{eqnarray*} 
where, for fixed $t>0$, $\overline \nu\in V$ $s\in[0,1]^d$ and $i\in\mathbb Z^d$, the value $u^{\e,s,t\overline \nu}_i$ is defined 
by
\begin{equation}\label{uexizd}
u^{\e,s,t\overline \nu}_i=u^{\e,s,t\overline \nu}(\e i)=u_\e\bigg(\e\sum_{\ell=1}^d t s_\ell\nu_\ell 
+\e \sum_{\ell=1}^d t i_\ell \nu_\ell\bigg)
\end{equation} 
and the set $U^\e_{s,t\overline \nu}$ is given by 
$$U^\e_{s,t\overline \nu}=\Big\{x\in\mathbb R^d: \e\sum_{\ell=1}^d t s_\ell \nu_\ell+t L(\overline u)(x)\in U\Big\},$$ 
with $L(\overline \nu)$ the linear map such that $L(\overline \nu)(e_\ell)=\nu_\ell$ for any $
\ell\in\{1,\dots, d\}$. 
Hence, from \eqref{stimainfd} we get the estimate 
\begin{equation}\label{stimainfd2}
F_\e(u_\e)\geq \frac{1}{2d}\frac{\mathcal H^{d-1}(S^{d-1})}{\mathcal H^{k_d}(V)} 
\int_{V}\int_0^{+\infty} \rho(t) t^{2d-1} \int_{[0,1]^d} X_\e(u^{\e,s, t\overline \nu}; U^\e_{s,t\overline\nu}) \, ds   \, dt\, d\mathcal H^{k_d}(\overline \nu).  
\end{equation}
Note moreover that, as in Remark \ref{indepp}, the functions $u^{\e,s,t \overline \nu}$ converge to the pull-back of the limit measure $\mu$ with respect to $t L(\overline \nu)$, so that 
$$\liminf_{\e\to 0} X_\e (u^{\e,s,t\overline \nu}; U^\e_{s,t\overline \nu})\geq 4\pi t^{2-d}\|M\|.$$

Now we can conclude the proof of the lower inequality, obtaining 
\begin{eqnarray*}
\liminf_{\e\to 0}F_\e(u_\e)&\geq& 
\frac{2\pi}{d}\frac{1}{\mathcal H^{k_d}(V)} 
\int_{V}
\int_0^{+\infty} \mathcal H^{d-1}(t S^{d-1}) \rho(t) t^{2}    \, dt\, d\mathcal H^{k_d}(\overline \nu)\ \|M\|(U)
\\
&\geq&  \frac{2\pi}{d}\int_{\mathbb R^d}\rho(|\xi|)|\xi|^2d\xi\,  \|M\|(U),
\end{eqnarray*} 
by exactly following the steps in the $2$-dimensional case. 
\end{proof}

\begin{proof}[Proof of Theorem {\rm\ref{comp-convergence}} {\rm (iii)} (upper bound)]
We subdivide the proof in two cases, by treating first the case of a single vortex of multiplicity $1$. 
\smallskip 

\noindent(a) {\em Upper bound for a single vortex of degree $1$.} \
We first deal with the case $d=2$, whose computations are then used in the general case.  

It is not restrictive to fix $\mu=\pi \delta_0$ and $\Omega=B_R$. 
In this case, we set $v(x)=\frac{x}{|x|}$ for $x\neq 0$, and we simply take $u_\e(x)=v(x)$ as a recovery sequence. 

Let $r_\e>0$ be such that 
$2\e<r_\e<<1$. 
Note that, since $|u_\e(x)-u_\e(y)|\leq 2$, then 
\begin{equation*}
\frac{1}{\e^2|\log\e|}\int_{B_1}\rho(|\xi|)\int_{B_{r_\e}} |u_\e(x+\e\xi)-u_\e(x)|^2\, dx \, d\xi\leq \frac{C r_\e^2}{\e^2|\log\e|}, 
\end{equation*}
so that we get  
\begin{eqnarray}\label{uev}
&&\hspace{-1cm}
\frac{1}{\e^2|\log\e|}\int_{B_1}\rho(|\xi|)\int_{B_{R}} |u_\e(x+\e\xi)-u_\e(x)|^2\, dx \, d\xi\nonumber\\
&&\leq \frac{1}{\e^2|\log\e|}\int_{B_1}\rho(|\xi|)\int_{B_{R}\setminus B_{r_\e}} |v(x+\e\xi)-v(x)|^2\, dx \, d\xi+\frac{C r_\e^2}{\e^2|\log\e|}. 
\end{eqnarray}

\bigskip 

Let $\xi\in B_1$ be fixed. Note that for $x\in B_{R}\setminus B_{r_\e}$ the value $u_\e(x+\e\xi)=\frac{x+\e\xi}{|x+\e\xi|}$ is well defined. 
Hence, we have 
\begin{eqnarray*}
&&\hspace{-1cm}\int_{B_{R}\setminus B_{r_\e}} | v(x+\e \xi)-v(x)|^2\, dx=\int_{B_{R}\setminus B_{r_\e}} \Big| \frac{x+\e \xi}{|x+\e \xi|}-\frac{x}{|x|} \Big|^2\, dx \\
&&= 2\int_{B_{R}\setminus B_{r_\e}} 
\frac{|x|}{|x+\e\xi|}\Big(\sqrt{1+2\e \frac{\langle x,\xi\rangle}{|x|^2}+\e^2\frac{|\xi|^2}{|x|^2}}- 1 -\e \frac{\langle x,\xi\rangle}{|x|^2}\Big)\, dx\\
&&\leq \int_{B_{R}\setminus B_{r_\e}} 
\frac{|x|}{|x+\e\xi|}\Big(
\e^2\frac{|\xi|^2}{|x|^2}-\e^2 \frac{|\langle x,\xi\rangle|^2}{|x|^4}+C\frac{\e^3}{|x|^3}\Big)\, dx\\
&&\leq\Big(1+\frac{\e}{r_\e}\Big) \int_{B_{R}\setminus B_{r_\e}} 
\Big(
\e^2\frac{|\xi|^2}{|x|^2}-\e^2 \frac{|\langle x,\xi\rangle|^2}{|x|^4}+C\frac{\e^3}{|x|^3}\Big)\, dx\\
&&=\Big(1+\frac{\e}{r_\e}\Big) \int_{B_{R}\setminus B_{r_\e}} 
\e^2\Big(
\frac{|\xi|^2}{|x|^2}-\frac{1}{2} \frac{|\langle e_1,\xi\rangle|^2+|\langle e_2,\xi\rangle|^2}{|x|^2}+C\frac{\e}{|x|^3}\Big)\, dx\\
&&=\Big(1+\frac{\e}{r_\e}\Big) \int_{B_{R}\setminus B_{r_\e}} 
\e^2\Big(
\frac{|\xi|^2}{2|x|^2}+C\frac{\e}{|x|^3}\Big)\, dx\\
&&=\Big(1+\frac{\e}{r_\e}\Big)2\pi\Big(\e^2|\xi|^2\int_{r_\e}^{R} \frac{1}{2r}\, dr +C\e^3 \int_{r_\e}^{R} \frac{1}{r^2}\, dr\Big)\\
&&=\Big(1+\frac{\e}{r_\e}\Big)2\pi\Big(\e^2|\xi|^2 \frac{1}{2}|\log r_\e-\log R| +C\e^3 (\frac{1}{r_\e}-\frac{1}{R})\Big)\\
&&\leq\Big(1+\frac{\e}{r_\e}\Big)\Big(|\xi|^2 \pi \e^2|\log r_\e| +2\pi C  \frac{\e^3}{r_\e}\Big)+ C\e^2|\log R|,  
\end{eqnarray*}
where as usual $C$ denotes a positive constant independent of $\e, x$ and $\xi$ and $R$.  
It follows that
\begin{eqnarray*}
&&\hspace{-1cm}\frac{1}{\e^2|\log\e|}\int_{B_1}\rho(|\xi|)\int_{B_{R}\setminus B_{r_\e}} \Big| \frac{x+\e \xi}{|x+\e \xi|}-\frac{x}{|x|} \Big|^2\, dx\, d\xi \\
&&\leq 
\frac{|\log r_\e|}{|\log\e|}\pi \int_{B_1}\rho(|\xi|)|\xi|^2\, d\xi  + C
\frac{\e}{r_\e}\frac{|\log r_\e|}{|\log\e|} 
+ 
C \Big(\frac{\e}{|\log\e|r_\e}
+\frac{\e^2}{|\log\e| r_\e^2}
+C\frac{|\log R|}{|\log \e|}\Big). 
\end{eqnarray*}
If we choose $r_\e$ such that 
$$
\lim_{\e\to 0}
\frac{ r_\e^2}{\e^2|\log\e|}=
\lim_{\e\to 0}\frac{\e}{r_\e}= 0 \ \ \hbox{\rm and } \ \ 
\lim_{\e\to 0}\frac{\log r_\e}{\log\e}= 1,$$ 
for example $r_\e=\e\log|\log\e|$, then, recalling \eqref{uev},  
\begin{eqnarray*}
&&\hspace{-1cm}
\limsup_{\e\to 0}\frac{1}{\e^2|\log\e|}\int_{B_1}\rho(|\xi|)\int_{B_{R}} | u_\e(x+\e \xi)-u_\e(x)|^2\, dx\, d\xi \\
&&\leq \limsup_{\e\to 0}\frac{1}{\e^2|\log\e|}\int_{B_1}\rho(|\xi|)\int_{B_{R}\setminus B_{r_\e}} \Big| \frac{x+\e \xi}{|x+\e \xi|}-\frac{x}{|x|} \Big|^2\, dx\, d\xi 
\\
&&\leq 
\pi \int_{B_1}\rho(|\xi|)|\xi|^2\, d\xi, 
\end{eqnarray*}
concluding the proof of the upper bound. 

\medskip

The two-dimensional computations can be used to treat the $d$-dimensional case when we have $M=\delta_0\times {\cal H}^{d-2}$ on a product $\Omega=B_{R}\times\omega$ with  $\omega\subset \mathbb R^{d-2}$.
In this case we take $u_\e(x)=u_\e(x_1,x_2)$ as defined in the case $d=2$. As in that case, the asymptotic behaviour of  $F_\e(u_\e)$ can be reduced to a computation involving $v(x_1,x_2)$, where $v(x)=x/|x|$.
For each $\xi\in \mathbb R^d$ the computation above gives
\begin{eqnarray*}
\int_{B_{R}\setminus B_{r_\e}} | v(x+\e \xi)-v(x)|^2\, dx&\leq&\Big(1+\frac{\e}{r_\e}\Big)(|\langle\xi,e_1\rangle|^2+|\langle\xi,e_2\rangle|^2) \pi \e^2|\log r_\e|\\
&&+C\e^2\Big(1+\frac{\e}{r_\e}\Big)\Big(\frac{\e}{r_\e}+|\log R|\Big),  
\end{eqnarray*}
so that, still following the computations above, we get 
\begin{eqnarray*}
\limsup_{\e\to 0}F_\e(u_\e)\le   {\mathcal L}^{d-2}(\omega)\pi \int_{B_1}\rho(|\xi|)(|\langle\xi,e_1\rangle|^2+|\langle\xi,e_2\rangle|^2)\, d\xi.  
\end{eqnarray*} 
Noting that 
$$
\int_{B_1}\rho(|\xi|)|\langle\xi,e_1\rangle|^2\, d\xi= \int_{B_1}\rho(|\xi|)|\langle\xi,e_1\rangle|^2\, d\xi={1\over d} \int_{B_1}\rho(|\xi|)|\xi|^2\, d\xi,
$$
we deduce that
\begin{equation}\label{gasup}
\Gamma\hbox{-}\limsup_{\e\to 0}F_\e(\mu)\le {\mathcal L}^{d-2}(\omega){2\pi\over d} \int_{B_1}\rho(|\xi|)|\xi|^2\, d\xi= C_\varrho \|M\|,  
\end{equation} 
where $C_\rho$ is defined as in \eqref{Crho}.

\bigskip 

\noindent(b) {\em The general case.}\ 
We consider a $(d-2)$-dimensional current $M$ such that spt$(M)\subset\subset\Omega$ and is a polyhedral boundary; that is, there exists a $(d-1)$-dimensional polyhedral current $L$ such that $\partial L=M$. Then, we can repeat the arguments for a single vortex up to considering in the place of $v(x)=\frac{x}{|x|}$ a function which plays the same role around each vortex. 
Indeed, 
by \cite[Theorem 9.6]{ABS} (see also \cite[Proof of Theorem 3]{AC}) there exist a map $\tilde u\colon\overline{\Omega}\to S^{1}$ and a finite union $S$ of $(d-3)$-dimensional simplices which contains all the $(d-3)$-dimensional faces of $L$, and $\delta,\gamma>0$ such that 
\begin{enumerate}
\item[(i)] $\tilde u\in W^{1,1}_{\rm loc}(\Omega;S^1)$ and $\star J\tilde u=M$; 
\item[(ii)] $\tilde u$ is locally Lipschitz in $\Omega\setminus(S\cup M)$ and there exists $p<\frac{3}{2}$ such that 
\begin{equation}\label{tildeu1} 
|D\tilde u(x)|=O\Big(\frac{1}{{\rm dist}(x, M)}\Big)+O\Big(\frac{1}{{\rm dist}(x, S)^p}\Big); 
\end{equation}
\item[(iii)] for every $(d-2)$-dimensional face $F$ of $M$, setting 
$$U(F,\delta,\gamma)=\Big\{x\in \Omega: {\rm dist}(x,F)\leq \min\Big\{\delta,\frac{\gamma}{\sqrt{1+\gamma^2}}\Big\}\, {\rm dist}(x,\partial F)\Big\},$$
then 
\begin{equation}\label{tildeu2}  
\tilde u(x)=\frac{x^\prime}{|x^\prime|} \ \ \hbox{\rm for }\ x\in U(F,\delta,\gamma), 
\end{equation} 
where we have identified the $(d-2)$-dimensional affine space containing $F$ with $\mathbb R^{d-2}$ and written $x\in \mathbb R^d$ as a pair $(x^\prime,x^{\prime\prime})\in \mathbb R^2\times \mathbb R^{d-2}$, omitting the dependence on $F$ (see \cite[Section 4.1]{ABS}). 
\end{enumerate}
In order to treat each singularity separately, we localize our energies as follows. 
For all $V\subset\Omega$, we set 
$$F_\e(u;V)={1\over\e^{d+2}|\log\e|} \int_{V\times V}\rho\Bigl({|x-y|\over\e}\Bigr) {|u(x)-u(y)|^2}\dxy.$$

Now, we take $u_\e(x)=\tilde u(x)$ for any $\e>0$. 
We can subdivide the computation of $F_\e(u_\e)$ into a computation ``close'' to $M$; that is, in each set $U(F,\delta,\gamma)$, which will give the main contribution, and in the rest, which will be small. 
Indeed, for each $F$, by \eqref{tildeu2} we can reduce to the computations above for a single vortex, obtaining 
\begin{equation}\label{stimaM} 
\limsup_{\e\to 0}F_\e(u_\e; U(F,\delta,\gamma))\leq C_\rho \|M\|(F). 
\end{equation} 
Now we have to prove that the contribution in the rest of the domain is negligible. 
To that end, we introduce the notation 
$$D^\eta=\{x\in\Omega: {\rm dist}(x,D)<\eta\}$$ 
for each $\eta>0$ and $D\subset\Omega$.  
With fixed $\eta>0$, let $U=U(\eta)$ be an open set of $\mathbb R^d$ such that 
$$M\cup S\subset U\subset\subset S^\eta\cup \bigcup_{F}U(F,\delta,\gamma),$$ 
where the union is taken over all $(d-2)$-dimensional faces of $M$. 
We first note that, since $\tilde u$ is Lipschitz in $\Omega\setminus \overline U$, we have 
\begin{equation}\label{stimafuori} 
\limsup_{\e\to 0}F_\e(u_\e; \Omega\setminus \overline U)=0. 
\end{equation}
Then, we consider the neighbourhood of the $(d-3)$-dimensional part of the boundary, 
showing that 
\begin{equation}\label{stimas} 
\limsup_{\e\to 0}F_\e(u_\e;S^\eta)=O(\eta) 
\end{equation} 
as $\eta\to 0$. This estimate can be achieved by estimating separately interactions close to $S$ (in a neighbourhood of order $\e$) and in the remaining of the set. 
To that end, let $T\geq 2$ be such that spt$(\rho)\subset B_T$; 
since $|u_\e(x)-u_\e(y)|\leq 2$, we have  
\begin{equation}\label{e}
F_\e(u_\e; S^\eta\cap (S\cup M)^{2T\e})\leq \frac{C}{|\log\e|}. 
\end{equation}
As for $S^\eta\setminus (S\cup M)^{(T+1)\e}$, we have the estimate 
\begin{eqnarray*}
&&\hskip-2.5cm F_\e(u_\e; S^\eta\setminus (S\cup M)^{(T+1)\e})\\
&=&
\frac{1}{|\log \e|}\int_{B_T}\rho(|\xi|)\int_{S^\eta\setminus (S\cup M)^{(T+1)\e}} 
\frac{1}{\e^2}|\tilde u(x+\e\xi)-\tilde u(x)|^2 \, dx\, d\xi\\ 
&\leq& 
\frac{1}{|\log \e|}\int_{B_T}\rho(|\xi|)\int_{S^\eta\setminus (S\cup M)^{(T+1)\e}} 
\int_0^1|D\tilde u(x+t\e\xi)|^2 |\xi|^2\, dt\, dx\, d\xi\\
&\leq&
\frac{1}{|\log \e|}\int_{B_T}\rho(|\xi|) |\xi|^2 
\int_0^1\int_{S^\eta\setminus (S\cup M)^{\e}} 
|D\tilde u(w)|^2\,  dw\,dt\,d\xi\\
&=&
\frac{1}{|\log \e|}\int_{B_T}\rho(|\xi|) |\xi|^2 \,d\xi
\int_{S^\eta\setminus (S\cup M)^{\e}} 
|D\tilde u(x)|^2\,  dx. 
\end{eqnarray*}
Now, by using estimate \eqref{tildeu1} we get 
\begin{eqnarray}\label{eta}
F_\e(u_\e; S^\eta\setminus (S\cup M)^{(T+1)\e})
&\leq& \frac{C }{|\log\e|}  
\int_{S^\eta\setminus (S\cup M)^{\e}} 
\Big(\frac{1}{{\rm dist}(x,S)^{2p}}+\frac{1}{{\rm dist}(x,M)^{2}}\Big)\,  dx\nonumber\\
&\leq& C \|M\|(S^\eta)   
\end{eqnarray}
(see \cite[Lemma 3]{AC}). 
By \eqref{e} and \eqref{eta} we obtain \eqref{stimas}. 

By gathering estimates \eqref{stimaM}, \eqref{stimafuori} and \eqref{stimas}, we can conclude 
that 
\begin{eqnarray}
\limsup_{\e\to 0}F_\e(u_\e)&\leq&\limsup_{\e\to 0}\Big(\sum_{F} F_\e(u_\e; U(F,\delta,\gamma))+F_\e(u_\e; \Omega\setminus \overline U)+F_\e(u_\e; S^{2\eta})\Big)\nonumber \\
&\leq& C_\rho \|M\|+O(\eta)\nonumber.  
\end{eqnarray}
By the arbitrariness of $\eta>0$, the claim is achieved. 

Finally, we can remove the hypothesis on $M$ by using the following density result (see {\em e.g.} \cite[Proposition 2]{AC} and \cite{A, ABS}). 

\begin{proposition}[density of polyhedral boundaries with multiplicity $1$]\label{polydens} 
Let $\Omega$ be a bounded domain in $\mathbb R^d$, and let $T$ be an integral boundary locally in $\Omega$ with dimension $h<d$ and $\|T\|<+\infty$. Then, there exists a sequence $\{T_n\}$ of polyhedral boundaries in $\mathbb R^d$ with multiplicity $1$ and {\rm spt}$(T_n)\subset\subset\Omega$ such that $\lim\limits_{n\to+\infty}{\bf F}_{U}(T_n-T)=0$ for any $U\subset\subset \Omega$ 
and $\lim\limits_{n\to+\infty}\|T_n\|=\|T\|$.  
\end{proposition}
Using this proposition, we can construct a recovery sequence by a diagonal argument, concluding the proof.  
\end{proof}

\noindent{\bf Acknowledgements.} The author gratefully acknowledges the University of Sassari for funding her research in the framework of the {\em Fondo di Ateneo per la ricerca 2020}. She is a member of the INdAM group GNAMPA.

\end{document}